\documentclass{amsart} 
\usepackage{amssymb}

\usepackage{mathpazo,courier}
\usepackage[scaled]{helvet}

\usepackage[usenames,dvipsnames]{xcolor}
\usepackage[pagebackref,colorlinks,linkcolor=BrickRed,citecolor=OliveGreen,urlcolor=Blue,hypertexnames=true]{hyperref}

\usepackage{enumerate}

\newcommand\N{\mathbb N}
\newcommand\Z{\mathbb Z}

\newcommand\R{\mathbb R}

\renewcommand\S{\mathbb S}
\newcommand\COP{\mathcal C}
\newcommand\PSD{\SYM^+}
\newcommand\MK{\mathcal{K}}

\newcommand\SYM{\mathcal{S}}

\newcommand\al\alpha
\newcommand\la\lambda
\newcommand\de\delta
\newcommand\De\Delta
\newcommand\ep\varepsilon
\newcommand\ph\varphi
\renewcommand\th\vartheta
\newcommand\si\Si
\newcommand\Si\Sigma
\newcommand\ta\tau
\newcommand\ps\psi
\newcommand\Ph\Phi
\newcommand\Ps\Psi

\theoremstyle{definition}
\newtheorem{thm}{Theorem}

\newtheorem{cor}[thm]{Corollary}
\newtheorem{pro}[thm]{Proposition}
\newtheorem{rem}[thm]{Remark}
\newtheorem{ex}[thm]{Example}

\newtheorem{lem}[thm]{Lemma}
\newtheorem{df}[thm]{Definition}

\newtheorem{conjecture}[thm]{Conjecture}

\title{Sum-of-squares certificates for copositivity via test states}

\author[Markus Schweighofer]{Markus Schweighofer}
\address{Fachbereich Mathematik und Statistik, Universität Konstanz, Germany}
\email{markus.schweighofer@uni-konstanz.de}

\author[Luis Felipe Vargas]{Luis Felipe Vargas}
\address{Centrum Wiskunde \& Informatica, Amsterdam, The Netherlands}
\email{lfv@cwi.nl}


\date{October 19, 2023}


\begin{document}
\begin{abstract}
\begin{itemize}
In 1995, Reznick showed an important variant of the obvious fact that any positive semidefinite (real) quadratic form is a sum of squares of linear forms: 
If a form (of arbitrary even degree) is positive definite then it becomes a sum of squares of forms after being multiplied by a sufficiently high
power of the sum of its squared variables. If the form is just positive \emph{semi}definite instead of positive definite, this fails badly in general.
In this work, we identify however two classes of positive semidefinite
even quartic forms for which the statement continues to hold even though they have in general infinitely many projective
real zeros. The first class consists of all even quartic positive semidefinite forms in five variables. This provides a natural certificate for a matrix of size five being copositive and
answers positively a question asked by Laurent and the second author in 2022.
The second class consists of certain quartic positive semidefinite forms that arise from graphs and their stability number.
This shows finite convergence of a hierarchy of semidefinite approximations for the stability number of a graph proposed by de Klerk and Pasechnik in 2002.
In both cases, the main tool for the proofs is the method of pure states on ideals developed by Burgdorf, Scheiderer and the first author in 2012.
We hope to make this method more accessible by introducing the notion of a \emph{test state}.
\end{itemize}
\end{abstract}

\maketitle

\section{Introduction}

We denote by $\N:=\{1,2,3,\ldots\}$, $\N_0:=\{0\}\cup\N$, $\R$ and $\R_{\ge0}$ the sets of natural, nonnegative integer, real and nonnegative real numbers, respectively.
Throughout the article, $x_1,\ldots,x_n$ denote formal variables, \[x:=(x_1,\dots,x_n)\] is the tuple formed by them and $\R[x]:=\R[x_1,\ldots,x_n]$ is the ring of real polynomials in
these variables. We denote by
\[\Si:=\left\{\sum_{i=1}^mp_i^2\mid m\in\N,\;p_1,\ldots,p_m\in\R[x]\right\}\]
the set of all sums of squares of polynomials. Every sum of squares of polynomials is of course globally nonnegative, that is, every polynomial from $\Si$ is nonnegative on $\R^n$.
In 1888 Hilbert knew already that the converse is false except for $n=1$ \cite{hilb}, it fails in fact dramatically as shown by Blekherman in 2006 \cite{ble}.
In 1927, Artin solved Hilbert's 17th problem affirmatively, that is, he showed that for any globally nonnegative polynomial $p\in\R[x]$ there exists a non-zero $h\in\R[x]$ such that
$h^2p\in\Si$ \cite{art}.

A polynomial in $\R[x]$ is called a \emph{form} if it is homogeneous, that is, all of its monomials have equal degree. It is called \emph{even} if it is of the form
$p(x_1^2,\ldots,x_n^2)$ for some $p\in\R[x]$. It is called \emph{linear}, \emph{quadratic}, \emph{cubic}, \emph{quartic} if it is of degree $1,2,3,4$, respectively.
A form is called \emph{positive semidefinite} or \emph{positive definite} if it is (pointwise) nonnegative on $\R^n$ or positive on $\R^n\setminus\{0\}$, respectively.
It is a simple fact from linear algebra that every positive semidefinite quadratic form is a sum of squares of linear forms.
In 1995, Reznick showed the following important variant of both this fact and Artin's theorem \cite[Theorem 3.12]{rez}
(Reznick proved a theorem that is much stronger in several
respects, but for us this popular weaker version is the appropriate statement).

\begin{thm}[Reznick]\label{reznick}
Let $p$ be a positive definite form. Then, there exists $r\in\N_0$ such that 
\begin{align}\label{cert-reznick}
\left(\sum_{i=1}^nx_i^2\right)^rp\in \Si.
\end{align}
\end{thm}

The hypothesis ``positive definite'' cannot in general be weakened to ``positive semidefinite'' in Theorem \ref{reznick}.
However, Scheiderer showed that every positive semidefinite form in three variables admits a nonnegativity certificate as in (\ref{cert-reznick})
\cite[Corollary 3.12]{sch1}. It is easy to show that this result does not extend to $n\ge 4$ (see for example \cite[Theorem 4.3]{cl} or \cite[Subsection 2.1]{vl}).

In this paper, we show the existence of a certificate as in (\ref{cert-reznick}) for certain positive semidefinite (but in general not positive definite) even quartic forms associated to
copositive matrices. We write $\SYM_n$ for the vector space of real symmetric matrices of size $n$. The set of \emph{copositive matrices} of size $n$
\begin{align}\label{def-COP}
\COP_n:=\{M\in \SYM_n\mid a^TMa\ge 0 \text{ for } a\in \R_{\ge0}^n\}
\end{align}
forms a (convex) cone in $\SYM_n$. While this cone looks innocently similar to the cone
\begin{align*}
\PSD_n:=\{M\in \SYM_n\mid a^TMa\ge 0 \text{ for } a\in \R^n\}\subseteq\COP_n
\end{align*}
of positive semidefinite matrices, it turns out that it maliciously captures many very difficult problems.
Indeed, it has been shown to have many applications in combinatorial optimization \cite{kp,bk}. Burer showed for example how to rewrite any quadratic optimization problem involving binary and continuous variables as a copositive optimization problem, i.e., as linear optimization problem over the cone $\COP_n$ \cite{bur}. Consequently, copositive optimization is hard in general. Moreover, the problem of determining whether a matrix is copositive is a co-NP-complete problem \cite{MK}. This motivates to study tractable certificates for copositivity. The certificates we will consider are based on the nonnegativity certificate from
\eqref{cert-reznick}.

Denoting $x^{\circ2}:=(x_1^2, \dots, x_n^2)$ and viewing this as a column vector, we call for each $M\in \SYM_n$,
\[(x^{\circ2})^TMx^{\circ2}\]
the \emph{(even) quartic form associated} to $M$. By means of this quartic form, we can obviously rewrite the definition \eqref{def-COP} of $\COP_n$ as 
\[\COP_n=\{M\in \SYM_n\mid (x^{\circ2})^TMx^{\circ2} \text{ is positive semidefinite}\}.\]

\begin{df}
Let $M\in \SYM_n$ be a symmetric matrix. We call $M$ \emph{Reznick-certifiable} if
\[\left(\sum_{i=1}^nx_i^2\right)^r(x^{\circ2})^TMx^{\circ2} \in \Si\]
for some $r\in \N_0$.
\end{df}

Using a different terminology, Reznick-certifiability of copositiveness was first considered by Parrilo in his thesis \cite{par}.
Later, it was used by de Klerk and Pasechnik who defined for $r\in\N_0$ the cones $\MK_n^{(r)}$ for approximating $\COP_n$ \cite[Section 4]{kp}:
\begin{align}\label{K}
\MK_n^{(r)}:=\Big\{M\in \SYM_n\mid\left(\sum_{i=1}^nx_i^2\right)^r(x^{\circ2})^TMx^{\circ2}\in \Si\Big\}.
\end{align} 
In other words, 
\[M\text{ is Reznick-certifiable} \iff M\in \bigcup_{r\in\N_0}\MK_n^{(r)}.\]
By Reznick's theorem (Theorem \ref{reznick}), we have that each $M\in\SYM_n$ satisfying
\begin{align}\label{strictly}
a^TMa>0\text{ for all }a\in\R_{\ge0}^n\setminus\{0\}
\end{align}
is Reznick-certifiable.

In this article, we will prove Reznick-certifiability for two classes of special copositive matrices: The first class is the set of all copositive matrices of size $5$
(for those that even satisfy \eqref{strictly} this follows already from Reznick's theorem). The second class consists of certain copositive matrices of size $n$
arising from a graph on $n$ vertices when one wants to compute the stability number of the graph via copositive optimization. These latter matrices never satisfy
\eqref{strictly}.

\subsection*{Copositive matrices of size $5$.}
 
The cone of $5\times 5$ copositive matrices has been much studied in the literature, and has been shown to play a special role. In 1962, it was shown by Diananda \cite{dia} that, for $n\le 4$, every $n\times n$ copositive matrix $M$ can be written as $M=P+N$, where $P$ is a positive semidefinite matrix and $N$ is a matrix with only nonnegative entries.
In 1976, Choi and Lam showed, for any $n\in\N$, that the cone $\MK_n^{(0)}$ consists precisely of the matrices that can be written as $P+N$ with $P$ positive semidefinite and $N$ entrywise nonnegative \cite[Lemma~3.5]{cl}. In particular, for $n\le 4$, every $n\times n$ copositive matrix lies in $\MK_n^{(0)}$ (that is, $\COP_n=\MK_n^{(0)}$).

This result does not extend to $n\ge 5$. Indeed, Hall and Newman \cite{hn} showed one year later that the \emph{Horn matrix}, defined as 
\begin{align}\label{horn}
H = \left(\begin{matrix}
1 & 1 & -1 & -1 & 1\cr
1 & 1 & 1 & -1 & -1\cr
-1 & 1 & 1 & 1 & -1\cr
-1 & -1 & 1 & 1 & 1 \cr
1 & -1 & -1 & 1 & 1
 \end{matrix}\right),
\end{align}
 is copositive and cannot be written as $P+N$, with $P$ positive semidefinite and $N$ entrywise nonnegative. This shows $\COP_5\ne\MK_5^{(0)}$. However, Parrilo showed that the Horn matrix satisfies
 \[H\in \MK_5^{(1)}\]
and therefore is Reznick-certifiable \cite[Page 68]{par}.  By an easy construction described in \cite[Lemma 15]{gl}, one can extend $H$ to a matrix
$H'$ of any wished size $n>5$ such that  $H'\in \MK_5^{(1)}\setminus\MK_5^{(0)}$. Hence $\COP_n\ne\MK_n^{(0)}$ for any $n\ge5$.

Laurent and the second author constructed, for any $n\ge6$, $n\times n$ copositive matrices that are not Reznick-certifiable \cite[Theorem 3]{lv2} (for example
the quartic form associated to the Horn matrix $H$ from \eqref{horn} when viewed as polynomial in $n\ge6$ instead of $5$ variables) and asked the question of whether every $5\times 5$ matrix is Reznick-certifiable \cite[Question 1]{lv2}. In this paper, we answer the latter question affirmatively. This is our first main result:

\begin{thm}\label{reznick5cop}
Every copositive matrix of size $5$ is Reznick-certifiable, in other words
\[\COP_5=\bigcup_{r\ge0}\MK_5^{(r)}.\]
\end{thm}

To prove this, we will use the important reduction of Laurent and the second author \cite[Theorem 1.3]{lv3} showing that it suffices to show that every positive diagonal scaling
$DHD$ of the Horn matrix $H$ with a positive definite diagonal matrix $D\in\SYM_5$ is Reznick-certifiable (see Theorem~\ref{theorem-reduction}).
We briefly explain this reduction.
The extreme rays of $\COP_5$ have been fully described by Hildebrand \cite{hild} and up to conjugation with permutation matrices (i.e., up to multiplying with a permutation matrix
from one side and its transpose from the other side) they can be divided intro three categories:
The first category consists of matrices that belong to cone $\MK_n^{(0)}$. The second category arises from a special class of matrices $T(\psi)$ \cite[Page 1539]{hild}
by positive diagonal scalings \cite[Theorem 3.1]{hild} and has been shown to be Reznick-certifiable in \cite{lv3}. The third category consists of the positive diagonal scalings of the Horn matrix. Since both cones $\COP_n$ and $\MK^{(r)}_n$ are obviously invariant under conjugation with permutation matrices, we can disregard the conjugation by permutation
matrices. While $\COP_n$ is obviously invariant also under positive diagonal scaling, this is never the case for $\MK^{(r)}_n$ when $n\ge5$ and $r\ge1$ (since
otherwise \cite[Lemma 1]{ddgh} would imply $\MK_n^{(r)}= \MK_n^{(0)}$ which we have remarked above to be false). In contrast to conjugation by permutation matrices,
we cannot ignore the positive diagonal scaling in Hildebrand's result. Quite to the contrary, given a matrix in $M\in\COP_n\setminus\MK_n^{(0)}$ and an $r\in\N$, there exists
a positive definite diagonal matrix $D$ of size $n$ such that $DMD\not\in\MK_n^{(r)}$ \cite{ddgh}. In particular, there is no $r\in\N$ such that $\COP_5=\MK_5^{(r)}$, i.e., the
union of the right hand side of the equation in Theorem~\ref{reznick5cop} needs to be infinite. This follows also from a much stronger result in the recent work
\cite{bkt} where it is shown that $\COP_5$ is not even the projection of a spectrahedron (whereas each $\MK^{(r)}_n$ obviously is) \cite[Corollary 3.18]{bkt}.

The matrices in the first category are trivially Reznick-certifiable. The matrices in the second category have been shown to be Reznick-certifiable by Laurent and the second author
\cite[Theorem 2.3]{lv3}. There it is crucially used that the quadratic form $x^TT(\ps)x$ vanishes on only finitely many rays inside the orthant $\R_{\ge0}^5$. The quadratic form
$x^THx$ vanishes however on infinitely many rays inside this orthant \cite[Page 40]{lv3} which made Theorem \ref{reznick5cop} inaccessible by the methods used in \cite{lv3}.
In this article, we manage to handle the third category by using the theory of pure states on ideals from \cite{bss}.

\subsection*{Copositive matrices arising from graphs.}

The second class of copositive matrices for which we prove Reznick-certifiability arises from graphs. By a graph, we mean a simple undirected loopless finite graph, that is a graph
is a pair $G=(V,E)$ where $V$ is a finite set (the set of vertices) and $E$ is a set of two-element subsets of $V$ (the set of edges), i.e., $E\subseteq\{\{i,j\}\mid i,j\in V,i\ne j\}$.
Here, we will often suppose without loss of generality that $V=[n]:=\{1,\ldots,n\}$ for some $n\in\N_0$. In this case, the adjacency matrix $A_G=(a_{ij})_{1\le i,j\le n}\in\SYM_n$
is defined by $a_{ij}=1$ if $\{i,j\}\in E$ and $a_{ij}=0$ if $\{i,j\}\not\in E$.
 A subset of vertices $S\subseteq V$ is \emph{stable} in $G$ if $\{i,j\}\notin E$ for all $i,j\in S$. The stability number of $G$, denoted by $\al(G)$, is the maximum cardinality of a stable set in $G$. Computing $\al(G)$ is an NP-hard problem in general \cite{kar}. De~Klerk and Pasechnik \cite[Corollary 2.4]{kp} proposed the following formulation of $\al(G)$ as an optimization problem over the copositive cone $\COP_n$, which can easily be deduced from \cite[Theorem 1]{ms}:
\begin{align}\label{alpha-cop}
\al(G)=\min\{ t\in\R \mid t(A_G+I)-J\in \COP_n\},
\end{align}
where $A_G$ is the adjacency matrix of $G$, and $I$ and $J$ are the identity matrix and the all ones matrix of size $n$, respectively.
By taking $t=\al(G)$, we obtain that the {\em graph matrix} of $G$,
\[M_G:=\al(G)(A_G+I)-J\]
is copositive. Thus, the {\em graph polynomial}
\[f_G:=(x^{\circ2})^TM_Gx^{\circ2}\]
is nonnegative. As an illustration, when $G$ is the 5-cycle, the graph matrix $M_G$ is precisely the Horn matrix $H$. Our second main result shows that the graph matrix of any graph is always Reznick-certifiable.

\begin{thm}\label{theorem-alpha}
For any graph $G$, the matrix $M_G$ is Reznick-certifiable, i.e., there exists $r\in \N_0$ such that 
\[\left(\sum_{i=1}^nx_i^2\right)^rf_G\in \Si.\]
\end{thm}

Observe that this result is not a direct consequence of Reznick's theorem (Theorem \ref{reznick}) because the polynomial $f_G$ has zeros. For example, if $S\subseteq V=[n]$ is a stable set of size $\al(G)$, then we have $f_G(x)=0$ for the characteristic vector $x\in\R^n$ of $S$ (defined by $x_i=1$ for $i\in S$ and $x_i=0$ for $i\in[n]\setminus S$). In general $f_G$ may have infinitely many zeros on the sphere
\[\S^{n-1}:=\{x\in\R^n\mid\|x\|=1\},\]
as shown in \cite[Corollary 4.4]{lv1}.
The study of the sum-of-squares certificates for the matrices $M_G$ is motivated by the convergence analysis of a hierarchy of semidefinite approximations for $\al(G)$ proposed by de Klerk and Pasechnik \cite[Section 4]{kp}. This hierarchy is obtained by replacing the cone $\COP_n$ by the cones $\MK_n^{(r)}$ ($r\in\N_0$)
in the formulation (\ref{alpha-cop}) for $\al(G)$:
\[\th^{(r)}(G):=\min\left\{t\mid t(A_G+I)-J\in \MK_n^{(r)}\right\}.\]
Hence, we have 
\[\th^{(r)}(G)=\al(G) \iff M_G\in \MK_n^{(r)} \iff \left(\sum_{i=1}^nx_i^2\right)^rf_G\in \Si.\]
It is obvious that \[\al(G) \le\dots\le \th^{(3)}(G) \le \th^{(2)}(G)\le\th^{(1)}(G)\le \th^{(0)}(G).\]
Using Reznick's Theorem \ref{reznick}, it is easy to show that for any fixed graph $G$,
\[\lim_{r\to\infty}\th^{(r)}(G)=\al(G).\]
De Klerk and Pasechnik conjectured that this hierarchy converges for each non-empty graph to $\al(G)$ after $\al(G)-1$ steps, i.e., $\th^{(\al(G)-1)}(G)=\al(G)$
\cite[Conjecture 5.1]{kp}. 

\begin{conjecture}[de Klerk and Pasechnik]\label{conjecture}
For any non-empty graph $G$, we have  \[\left(\sum_{i=1}^nx_i^2\right)^{\al(G)-1}f_G\in \Si.\] In other words, $M_G\in \MK_n^{(\al(G)-1)}$.
\end{conjecture}

Conjecture \ref{conjecture} is known to hold for perfect graphs \cite[Lemma 5.2]{kp} (see also \cite[Lemma 4]{gl}), for graphs with $\al(G)\le 8$ \cite[Corollary 1]{gl} (see also \cite[Corollary~7]{pvz} for $\al(G)\le 6$) and for cycles and their complements \cite[Corollaries 5.4 and 5.6]{kp}. Our result (Theorem~\ref{theorem-alpha}) shows the finite convergence of the parameters $\th^{(r)}(G)$ to $\al(G)$. The conjecture by de Klerk and Pasechnik remains open.

\subsection*{Overview.} Our article can be seen as case study of the use of the theory of
pure states on ideals developed by Burgdorf, Scheiderer and Schweighofer in 2012 \cite{bss}. These pure states have been introduced as a tool to prove membership in
so-called \emph{quadratic modules}. We think that this tool did not receive enough attention and has a lot of potential. In Section \ref{sec:qm}, we will recall the notion of a (Archimedean) quadratic module of the polynomial ring $\R[x]$ and the most relevant facts about it, in particular how it is related to Reznick's theorem.
Section \ref{sec:pure} recalls the machinery of pure states from \cite{bss}. 
A hopefully more accessible version of this machinery will be presented in Section \ref{sec:test} where
we introduce the new notion of a \emph{test state}. The reader who skips the proofs in Section \ref{sec:test} can readily skip Section \ref{sec:pure}. Finally Sections \ref{sec:5}
and \ref{sec:stable} are devoted to the proofs of our main results, Theorem \ref{reznick5cop} and \ref{theorem-alpha}, respectively.
\\
\\
The main results of this article were included in the PhD thesis of the second author \cite{var}.


\section{Review of quadratic modules}\label{sec:qm}

For elements $a$ and $b$ and subsets $A$ and $B$ of the same ring (in this section $\R[x]$), we use self-explanatory notation such as
$AB:=\{ab\mid a\in A,b\in B\}$, $a+B:=\{a+b\mid b\in B\}$ and $Ab:=\{ab\mid a\in A\}$.

\begin{df}\label{def-qm}
Let $M\subseteq\R[x]$.
\begin{enumerate}[(a)]
\item We call
\[S(M):=\{a\in\R^n\mid p(a)\ge0\text{ for all }p\in M\}\]
the \emph{nonnegativity set} of $M$.
\item A subset $M$ of $\R[x]$ is called a \emph{quadratic module} of $\R[x]$ if \[1\in M,\quad M+M\subseteq M\quad\text{and}\quad\Si M\subseteq M.\]
\item A quadratic module $M$ is called \emph{Archimedean} if $M+\Z=\R[x]$.
\end{enumerate}
\end{df}

The following result is folklore \cite[Lemma 4.3.4]{schw}.

\begin{pro}\label{archchar}
Let $M$ be a quadratic module of $\R[x]$. Then, the following assertions are equivalent:
\begin{enumerate}[(a)]
\item $M$ is Archimedean.
\item There exists $N\in \N$ such that  $N-\sum_{i=1}^nx_i^2\in M$.
\end{enumerate}
\end{pro}

\begin{ex}\label{sphere}
$M_{\S^{n-1}}:=\Si+\R[x](1-\sum_{i=1}^nx_i^2)$ is an Archimedean quadratic module of $\R[x]$ with the unit sphere as nonnegativity set:
 \[S\left(M_{\S^{n-1}}\right)=\S^{n-1}.\]
\end{ex}

The following result of de Klerk, Laurent and Parrilo from 2005 \cite[Proposition 2]{klp} will be very important for us:

\begin{pro}[de Klerk, Laurent and Parrilo]\label{reznick-putinar-sphere}
For every form $p\in\R[x]$ of even degree, the following are equivalent:
\begin{enumerate}[(a)]
\item $p$ satisfies \eqref{cert-reznick}, i.e., $\left(\sum_{i=1}^nx_i^2\right)^rp\in\Si$ for some $r\in\N_0$.
\item $p\in M_{\S^{(n-1)}}$
\end{enumerate}
\end{pro}

From this result, we will actually need only the weaker form below. We will need it to prove Lemma \ref{isolated-nodes}, which will be an important
ingredient to prove Theorem \ref{theorem-alpha}.

\begin{cor}[de Klerk, Laurent and Parrilo]\label{equiv-cert}
For every $M\in \SYM_n$, the following are equivalent:
\begin{enumerate}[(a)]
\item $M$ is Reznick-certifiable.
\item $(x^{\circ2})^TMx^{\circ2}\in M_{\S^{(n-1)}}$
\end{enumerate}
\end{cor}

Now that we have introduced quadratic modules and have announced that the particular quadratic module $M_{\S^{(n-1)}}$ will be important for us, we should
state the most popular result about quadratic modules, namely Putinar's theorem from 1993 \cite[Theorem 1.2]{put}
(here in an insignificantly stronger version that is covered for example by \cite[Theorem 4]{jac} or \cite[Corollary 8.2.11 together with Remark 8.2.12]{schw}).

For a polynomial $p\in\R[x]$ and a set $S\subseteq\R^n$,
we write ``$p>0$ on $S$'' to express that $p$ is pointwise positive on $S$, i.e., $p(a)>0$ for all $a\in S$. Analogously, ``$p\ge0$ on $S$'' has the obvious meaning.

\begin{thm}[Putinar]\label{putinar}
Let $M$ be an Archimedian quadratic module and $p\in\R[x]$. Then
\[p>0\text{ on }S(M)\implies p\in M.\]
\end{thm}

It is an important topic in the literature under which additional hypotheses in Putinar's theorem
the condition \[p>0\text{ on }S(M)\] can be weakened to \[p\ge0\text{ on }S(M)\] (which is of course necessary for $p$ being contained in $M$).
We refer to \cite[Section 3]{sch2}, \cite[Section 3]{sch1}, \cite[Theorem 1.1]{nie}, \cite[Corollary 9.2.6]{schw} and the references therein.
One approach to this problem that we will pursue here is the theory of pure states on ideals developed by Burgdorf, Scheiderer and Schweighofer in 2012
\cite{bss}. We will recall this approach in Section \ref{sec:pure}.

The rough idea is that $p$ needs to be only nonnegative (instead of positive) at a point $a$ of $S(M)$
if it passes a number of tests at $a$. Such a test can often be that some (possibly higher) directional derivative of $p$ is positive at $a$ (see for example
\cite[Theorem 7.8]{bss}, \cite[Examples 8.3.4--8.3.6, Theorem 9.1.12]{schw}). Especially in the case where $p$ has infinitely many zeros on $S(M)$, the tests that have to be passed
at $a$ are however usually less of geometric than of algebraic nature.
This article can be seen as a case study that shows how to deal with such tests of rather algebraic nature. Formally, the tests
consist in testing positivity under so-called \emph{pure states} \cite{bss}. Their definition is complicated and we hope to make the method more popular by introducing the
much more concrete notion of a \emph{test state} in Definition~\ref{def-test-state} below.

As stated in Remark \ref{provesputinar} below, Putinar's theorem follows immediately from a very special case of the theory of pure states exposed in Section \ref{sec:pure}
below. The only case of Putinar's theorem that is directly relevant to us
is however the following special case of a weaker result of Cassier from 1984 \cite[Théorème 4]{cas}. We present
Cassier's theorem here as a corollary although it has been discovered much earlier than Putinar's theorem.

\begin{cor}[Cassier]\label{cassier}
Let $p\in\R[x]$. Then
\[p>0\text{ on }\S^{(n-1)}\implies p\in M_{\S^{(n-1)}}.\]
\end{cor}

By Proposition \ref{reznick-putinar-sphere}, one can easily deduce the special case of Corollary \ref{cassier} where $p$ is a form
from Theorem \ref{reznick} and vice versa. By the same proposition, any example showing that Reznick's certificate \eqref{cert-reznick}
does not extend to all positive \emph{semi}definite forms shows also that Corollary \ref{cassier} does not hold in general with ``$>$'' replaced by ``$\ge$''.

\section{Review of pure states on ideals}\label{sec:pure}

The notion of a \emph{state} originally comes from quantum physics. In the operator-theoretic approach to quantum physics, a state is a positive unital linear functional on a
$C^*$-algebra \cite[Subsection 4.5]{ba}. Our setting still shares the term ``state'' and also the fact that it is hard to capture by our traditional thinking but otherwise is extremely
different:
\begin{itemize}
\item Whereas $C^*$-algebras are most interesting when they are non-commutative we work here with commutative (unital) rings.
\item $C^*$-algebras are complete whereas for us the polynomial ring $\R[x]$ will be very important which is far from being complete in many senses
(for example with respect to the maximum norm on the unit ball of $\R^n$).
\item Positivity in our case will be understood in a certain formal sense. More precisely, we will require a state to map only ring elements \emph{possessing certain
nonnegativity certificates}, built upon sums of squares, to be mapped into the nonnegative reals.
\item The linear functionals will be defined only on an ideal of the commutative ring which in general will not contain $1$. In particular, the linear functionals cannot be unital
(i.e., cannot map the unital element $1$ of the ring to the real number $1$). However, we will require to have a good substitute for the
unital element $1$ of the ring which we will usually denote by $u$.
\end{itemize}
In the literature, one can find several occurrences where the original notion of state is looked at in a slightly more general or different context, including some of
the aspects just mentioned. See for example the work of Krivine \cite[Théorème 15]{kri} and most notably of Handelman \cite[Proposition 1.2]{han}. What the work in
Handelman still lacks is that he does not work with \emph{sum-of-squares based} certificates. The reconciliation of Handelman's setting with the theory of
sums of squares is the main difficulty in the work of Burgdorf, Scheiderer and the first author \cite{bss} from 2012.
The aim of this section is to introduce the reader briefly to this work. In Section \ref{sec:test}, we will however introduce the notion of a \emph{test state} which is a compromise
between the notions of a \emph{state} and a \emph{pure state} which will hopefully make the theory more accessible. This section will only be needed for the proofs in
Section \ref{sec:test} and conversely is heavily based on \cite{bss}. The reader who wants to see the theory in \cite{bss} from
the new angle provided by \emph{test states}, without wanting to see the corresponding proofs, can skip this section. The reader interested in the proofs or more examples
is referred to \cite{bss} or \cite[Chapter 7]{schw}.

\begin{df}
Let $V$ be a real vector space. We call a subset $C$ of $V$ a \emph{cone} (of $V$)
if $0\in C$, $C+C\subseteq C$ and $\R_{\ge0}C\subseteq C$. In this case, we call an element $u$ of $C$
a \emph{unit} of the cone $C$ (in $V$), if $C+\Z u=V$.
\end{df}

What we simply call a cone and a unit is often called \emph{order unit} and \emph{convex cone} in the literature \cite[Page 118]{bss}. Intuitively, a unit of a cone
is a kind of yardstick one can use to measure a kind of distance to the cone.

\begin{ex}\label{ex1}
Let $M$ be a quadratic module of $\R[x]$. Then $M$ is a cone that is Archimedean if and only if $u:=1\in\R[x]$ is a unit for $M$.
\end{ex}

Already in this very general framework
from convex geometry, one can define the notion of a state as follows. This definition is mostly employed in the case where $C$ is a convex cone with unit $u$.

\begin{df}\label{def-state}
Let $V$ be a real vector space, $C\subseteq V$ and $u\in V$. Then a linear function $\ph\colon V\to \R$ is called a \emph{state} of $(V,C,u)$ if $\ph(C)\subseteq\R_{\ge0}$
and $\ph(u)=1$. We denote the set of all states of $(V,C,u)$ by $S(V,C,u)$ and call it the \emph{state space} of $(V,C,u)$.
A state $\ph\in S(V,C,u)$ is called a \emph{pure state} of $(V,C,u)$ if whenever \[\ph=\la\ph_1+(1-\la)\ph_2\]
for some states $\ph_1,\ph_2\in S(V,C,u)$ and
some $\la\in\R$ with $0<\la<1$ then \[\ph=\ph_1=\ph_2.\]
\end{df}

Readers that are acquainted with basic convex geometry, notice of course that in the situation of the preceding definition
$S(V,C,u)$ is a convex subset of the vector space that is dual to $V$ and that the pure states of $(V,C,u)$ are by definition
just the extreme points of this convex set.

We now recap the following very important criterion for showing that a vector $v\in V$ belongs to the cone $C$. Its proof can be based on Zorn's lemma,
Tychonoff's theorem and the Krein-Milman theorem. To our knowledge it first appears in \cite[Theorem 1.4]{EHS}, see also \cite[Theorem 7.3.19]{schw}.

\begin{thm}[Effros, Handelman and Shen] \label{theorem-cone}
Suppose $u$ is a unit for the cone $C$ in the real vector space $V$ and let $v\in V$. If $\ph(v)>0$ for all pure states $\ph$ of $(V,C,u)$, then $v$ is also a unit for the cone $C$.
In particular, there exists $\ep>0$ such that $v-\ep u\in C$ and thus $v\in C$.
\end{thm}

We will apply Theorem \ref{theorem-cone} only in the proof of Theorem \ref{main-pure-states} below. 
Of course, it might be good for certain applications that the conclusion of this theorem is even that $v$ is a unit for $C$ rather than just $v\in C$. For us, this will however be rather a bad thing: It means that the theorem
can certainly not be applied to prove membership of elements that are not units of the cone. Because of Proposition \ref{reznick-putinar-sphere}, we will be interested in proving
membership of polynomials that are not positive on the sphere $\S^{n-1}$ in the quadratic module $M_{\S^{n-1}}$ (seen as a cone of the real vector space $\R[x]$).
Although $1$ is a unit for $M_{\S^{n-1}}$ by Examples \ref{sphere} and \ref{ex1}, it will therefore not work to work to apply the above theorem with $u:=1$. We will find different choices of $u$ that will however entail that we also will have to pass over from $\R[x]$ to a smaller subspace
(in order for $u$ continuing to be a unit) which will
actually be a proper ideal of the ring $\R[x]$. Consequently, we will have to work with a subcone of $M_{\S^{n-1}}$ (since $1\notin M_{\S^{n-1}}\setminus I$) which will turn
out to fulfill some algebraic closedness
properties captured by the following definition \cite{bss,schw}.

\begin{df}
Let $A$ be a commutative ring. The subset $T\subseteq A$ is called a \emph{preorder} of $A$ if   $A^2:=\{a^2\mid a\in A\} \subseteq T$, $T+T\subseteq T$ and $TT\subseteq T$.  Given a preorder $T$ of $A$, we say that $M\subseteq A$ is a \emph{$T$-module} of $A$ if $0\in M$, $M+M\subseteq M$, and $TM\subseteq M$. 
\end{df}

Note that we do not require $1\in M$ in the preceding definition.

\begin{ex}
Of course, $\Si$ is the smallest preorder of $\R[x]$. For any $M\subseteq\R[x]$,
$M$ is a quadratic module if and only if $1\in M$ and $M$ is a $\Si$-module of $\R[x]$.
\end{ex}

We now state a special case of
the \emph{dichotomy theorem} \cite[Corollary 4.12]{bss}  (see also \cite[Theorem 8.3.2]{schw}) which will be the second important ingredient
in the proof of Theorem \ref{theorem-cone} below. It divides the pure states of certain cones into two classes.

\begin{thm}[Burgdorf, Scheiderer and Schweighofer]\label{dichotomy}
Let $A$ be a commutative ring containing $\R$ as a subring. Suppose that $I$ is an ideal of $A$,  $T$ is a preorder of $A$, $M\subseteq I$ is a $T$-module of $A$, $u$ is a unit of $M$ in $I$, and $\ph$ is a pure state of $(I,M,u)$. Then exactly one of the following two assertions holds.
\begin{enumerate}[(I)]
\item $\ph$ is the restriction of a scaled ring homomorphism: There exists a ring homomorphism $\Ph\colon A\to \R$ such that $\Ph(u)\ne 0$ and \[\ph=\frac{1}{\Ph(u)}\Ph|_I.\]
\item There exists a ring homomorphism $\Ph\colon A\to \R$ with $\Ph|_I=0$ such that 
\[\ph(ab)=\Ph(a)\ph(b)\]
for all $a\in A$ and $b\in I$.
\end{enumerate}
\end{thm}

In the situation where $A=\R[x]$, it is easy to see that ring homomorphisms from $A$ to $\R$ are exactly the point evaluations in points of $\R^n$. More generally, states of type (I)
are in general easy to understand whereas pure states of type (II) remain often mysterious. The condition defining them is however easy to remember and will re-appear in
disguise in Definition \ref{def-test-state} below, where we will introduce the new notion of a \emph{test state}, which will allow to forget about states and pure states, at least for
our purposes.

\section{Test states and membership in quadratic modules}\label{sec:test}

In this section, we will introduce \emph{test states} and provide them
as a new tool to prove membership in Archimedian quadratic modules. While test states will remain mysterious in many
cases, they are easier to understand than the pure states from the preceding section that will only be important for the proofs in this section. For most purposes, it will probably
be enough to study test states rather than pure states. At least this will be the case for the proof of our two main results in the subsequent sections.

Test states come into play when the application of Putinar's theorem (Theorem~\ref{putinar}) is not possible because the polynomial $p$ for which one would
like to prove membership in the quadratic module $M$ has zeros on $S(M)$. Other techniques (pioneered to a large extent by Scheiderer
\cite{sch1,sch2}) might work as well in this case but come soon to their limits in the case where $p$ has infinitely many zeros on $S(M)$. In the case where
$p$ has finitely many zeros on $S(M)$, extensions of Putinar's theorem are known that still work when commonly known first and second order sufficient criteria
for (strict) local minima are satisfied by $p$ \cite{nie}. In this special case, pure states and test states would typically yield similar extensions since they typically
turn out to be related to (higher) derivatives at points of $S(M)$ (see Example \ref{derex} below or \cite[Examples 8.3.4--8.3.6]{schw}). The case when $p$ has infinitely many zeros
on $S(M)$ is much more difficult to handle and remains mysterious. In this case, we suspect that test states cannot be understood in purely geometric terms but rather
lead to conditions that are somehow related to the zeros of $p$ on $S(M)$ and yet are of algebraic nature. The concrete definition is
motivated by Condition (II) in the dichotomy theorem (Theorem \ref{dichotomy}), and is as follows.

\begin{df}\label{def-test-state}
Let $I$ be an ideal and $M$ be a quadratic module of $\R[x]$. Let $u\in I$ and $a\in\R^n$.
We call $\ph\in S(I,M\cap I,u)$ a \emph{test state} on $I$ for $M$ at $a$ with respect to $u$ if
\[\ph(pq)=p(a)\ph(q)\]
for all $p\in\R[x]$ and $q\in I$. We denote by \[T(I,M,u)_a\] the set of test states on $I$ for $M$ at $a$ with respect to $u$.
\end{df}

This definition might be hard to understand at first glance and we start by considering the case where $I=\R[x]$ and $u=1$.

\begin{ex}
Let $M$ be a quadratic module of $\R[x]$ and $a\in\R^n$. If $\ph$ is a test state on $\R[x]$ for $M$ at $a$ with respect to $1$. then we have 
\[\ph(p)=\ph(p\cdot 1)=p(a)\ph(1)=p(a) \text{ for all }p\in \R[x].\]
Therefore the only test state on $\R[x]$ for $M$ at $a$ with respect to $1$ is the evaluation at $a$
\[\R[x]\to\R,\ p\mapsto p(a).\]
\end{ex}

Before we see more examples, we introduce the following useful notation.

\begin{df} For a given polynomial $f\in\R[x]$, we denote by
\[Z(f):=\{a\in\R^n\mid f(a)=0\}\]
its (real) \emph{zero set}.
For each $F\subseteq\R[x]$, we denote by
\[I(F):=\left\{\sum_{i=1}^mg_if_i\mid m\in\N_0,g_1,\ldots,g_m\in\R[x]\right\}\]
 the ideal generated by $F$.
\end{df}

Now we come to a simple example of a test state that is not defined on the whole polynomial ring.

\begin{ex}\label{derex}
Suppose $n=1$ so that $\R[x]=\R[x_1]$. Let $M$ be a quadratic module of $\R[x]$ and $k\in\N_0$.
The only function that could possibly be a test state
on the ideal $I(x^k)=x^k\R[x]$ for $M$ at $0$ with respect to $x^k$ is
\[x^kR[x]\to\R,\ p\mapsto\frac1{k!}p^{(k)}(0),\]
that is, up to scaling, the $k$-th derivative at $0$.
Whether this function actually is a such a test state, depends on whether $\ph(M\cap I)\subseteq\R_{\ge0}$. If $S(M)$ contains some interval $[0,\ep]$ for some $\ep>0$ this is the case.
If $k$ is even, then it is also the case if $S(M)$ contains some interval $[-\ep,0]$ for some $\ep>0$.
\end{ex}

The following remark will be useful in the proof of Lemma \ref{isolated-nodes}.

\begin{rem}\label{rem-support}
Let $\ph$ be a test state on the ideal $I$ of $\R[x]$ for the quadratic module $M$ of $\R[x]$ at $a\in\R^n$ with respect to $u\in I$.
Suppose that $uM\subseteq M$. Then \[p(a)=p(a)\ph(u)=\ph(pu)=\ph(up)\ge0\] for all $p\in M$. Hence $a\in S(M)$, and
$a\in Z(f)$ for all $f\in M\cap(-M)$.
\end{rem}

Before we come to a strengthening of Putinar's theorem (Theorem \ref{putinar}) that involves test states, we need some preparation.

\begin{df} Let $V$ be a real vector space, $C\subseteq V$ a cone, $u\in V$ and $\emptyset\ne F\subseteq V$.
We say that $u$ is \emph{$F$-stably contained} in $C$ if for all $f\in F$ there exists a real $\ep>0$ such that $u+\ep f\in C$ and $u-\ep f\in C$.
\end{df}

Note that in the situation of the above definition, every element $F$-stably contained in $C$ is actually contained in $C$.

The following important lemma is essentially covered by \cite[Proposition 3.2]{bss} (see also \cite[Proposition 8.1.12]{schw}). Since these references use quite different
notation, we include the proof for convenience.

\begin{lem}\label{prop-unit}
Let $F\subseteq\R[x]$ be a nonempty set that generates the ideal $I:=I(F)$. Let $M$ be an Archimedean quadratic module of $\R[x]$ and $u\in \R[x]$ be a polynomial such that $uM\subseteq M$ such that $u$ is $F$-stably contained in $M$.
Then $u$ is even $I$-stably contained in $M$. In particular, if $u\in I$, then $u$ is a unit of the cone $I\cap M$ in the vector space $I$.
\end{lem}

\begin{proof}
It suffices to show that
\[J:=\{p\in \R[x]\mid \text{ there exists $\ep>0$ such that $u\pm \ep p\in M$}\}.\]
is an ideal of $\R[x]$. Indeed, the condition ``$u$ $F$-stably contained in $M$'' means $F\subseteq J$, while ``$u$ $I$-stably contained in $M$'' means $I\subseteq J$.

Clearly, $p\in J$ if and only if $-p\in J$. Now, if $\de>0$ and $\ep>0$ such that $u\pm \de p\in M$ and $u\pm \ep q\in M$ then $(\frac{1}{\de}+\frac{1}{\ep})u\pm (p-q) \in M$. Hence, 
 \[u\pm \frac{1}{\frac{1}\de+\frac{1}\ep}(p-q) \in M.\] 
 This shows that $p- q\in J$ if $p,q\in J$. 
 
 We finally show that if $p\in J$, then $pq\in J$ for all $q\in \R[x]$. For this, we observe that the following identity holds 
 \[q = \frac{1}{4}((q+1)^2-(q-1)^2).\]
 Then, it suffices to show that $pq^2\in J$ for all $q\in \R[x]$. Since $M$ is Archimedean and $p\in J$, there exists $N>0$ such that $N-q^2\in M$ and $Nu\pm p\in M$. Since $uM\subseteq M$, we have $Nu-uq^2\in M$. Since $\Si M \subseteq M$, we have $Nuq^2\pm pq^2\in M$. Hence,
\[ N^2u \pm pq^2=(N^2u-Nuq^2)+(Nuq^2\pm pq^2)\in M+M\subseteq M,\] 
as desired.
\end{proof}

Now we come to our membership criterion for quadratic modules which looks extremely technical, but will turn out to be very useful.

\begin{thm}\label{main-pure-states}
Let $F\subseteq \R[x]$ be a nonempty set of polynomials generating the ideal $I:=I(F)$. Let $M$ be an Archimedean quadratic module of $\R[x]$ and $f,u\in I$.  Suppose that
\begin{enumerate}[(a)]
\item $f\ge 0$ on $S(M)$,
\item $Z(f)\cap S(M)\subseteq Z(u)\cap S(M)$,
\item $uM\subseteq M$,
\item $u$ is $F$-stably contained in $M$, and
\item $\ph(f)>0$ for all $a\in Z(f)\cap S(M)$ and all $\ph\in T(I,M,u)_a$. 
\end{enumerate}
Then, there is $\ep>0$ such that $f-\ep u\in M$. In particular, $f\in M$.
\end{thm}

\begin{proof}
We will apply Theorem \ref{theorem-cone} in the following setting: The vector space is the ideal $I$, and the cone is $M\cap I$. In view of Lemma \ref{prop-unit}, using assumptions (c) and (d), we have that $u$ is a unit of $I\cap M$ in $I$. So let $\ph$ be a pure state of $(I,I\cap M, u)$. We show that $\ph(f)>0$. 

To this end, we apply Theorem \ref{dichotomy} in the following setting: The ring is $\R[x]$, the ideal is $I$, the preorder is $T:=\Si + u\Si$, the $T$-module is $M\cap I$
(use that $uM\subseteq M$), and the unit is $u$. The dichotomy theorem (Theorem~\ref{dichotomy}) now says that one of the following two alternatives holds:
\begin{enumerate}[(I)]
\item $\ph$ is the restriction of a scaled ring homomorphism: There exists a ring homomorphism $\Ph\colon\R[x] \to \R$ such that $\Ph(u)\ne 0$ and $\ph=\frac{1}{\Ph(u)}\Ph|_I$.
\item There exists a ring homomorphism $\Ph\colon\R[x]\to \R$ with $\Ph|_I=0$ such that \[ \ph(pq)=\Ph(p)\ph(q) \text{ for all } p\in \R[x]\text{ and }q\in I.\]
\end{enumerate}
It is easy to observe that every ring homomorphism $\Ph\colon\R[x]\to \R$ is given by a point evaluation, i.e., there exists $a\in \R^n$, such that $\Ph(p)=p(a)$ for all $p\in\R[x]$.
Therefore we find $a\in\R^n$ such that either
\begin{enumerate}[(I)]
\item $u(a)\ne0$, and $\ph(p)=\frac{1}{u(a)}p(a)$ for all $p\in I$, or
\item $u(a)=0$, and $\ph(pq)=p(a)\ph(q)$ for all $p\in \R[x]$ and $q\in I$.
\end{enumerate}
We claim that $a\in S(M)$. Indeed, let $p\in M$. We have $pu \in I\cap M$ due to (c). Then, (in both cases) we have $\ph(pu) = p(a)\ge 0$.

In case (I) we have $u(a)>0$ since $u\in M$ and $a\in S(M)$. This fact, together with (a) and (b), implies
$f(a)>0$, showing $\ph(f)=\frac{f(a)}{u(a)}>0$.
 
Finally, suppose that we are in the case (II). Then $\ph$ is a test state for $M$ at $a$ with respect to $u$, i.e., $\ph\in T(I,M,u)_a$.
Since $u,f\in I$, we can compute $\ph(fu)$ in two ways. First, we have $\ph(fu)=u(a)\ph(f)=0$. Also, $\ph(fu)=f(a)\ph(u)=f(a)$. Hence, $f(a)=0$, so $a\in Z(f)$, so that $a\in Z(f)\cap S(M)$. Since $\ph\in T(I,M,u)_a$, we must have that $\ph(f)>0$ by (e). 

So we arrive in both cases at the desired conclusion $\ph(f)>0$. Now applying Theorem \ref{theorem-cone}, we conclude the existence of $\ep>0$ such that
$f-\ep u\in M\cap I\subseteq M$.
\end{proof}

\begin{rem}\label{provesputinar}
If one sets $F:=\{1\}$ and $u:=1$ in Theorem \ref{main-pure-states}, then Condition (b) becomes $Z(f)\cap S(M)=\emptyset$ so that (e) is trivially satisfied, just as (c) and (d) are
in this case. Then Theorem \ref{main-pure-states} collapses to Putinar's theorem (Theorem \ref{putinar}).
\end{rem}

\begin{rem}\label{iremark}
Under the hypotheses of Theorem \ref{main-pure-states}, one sees easily that Condition (d) entails $Z(u)\cap S(M)\subseteq Z(p)\cap S(M)$ for all $p\in F$, so that
\begin{align}\label{realvariety}
S(M)\cap Z(u)\subseteq S(M)\cap\bigcap_{p\in F}Z(p)\overset{I=I(F)}=S(M)\cap\bigcap_{p\in I}Z(p)\overset{u\in I}\subseteq S(M)\cap Z(u).
\end{align}
It follows that every inclusion in \eqref{realvariety} can be replaced by an equality. In particular, $Z(u)$ and the real part of the affine variety defined by the ideal $I$
\[\bigcap_{p\in I}Z(p)\]
agree when intersected with $S(M)$. This gives some hint on how to choose $I$.
\end{rem}

Note that \eqref{realvariety} implies $S(M)\cap Z(u)\subseteq S(M)\cap Z(f)$ since $f\in I$ which the inclusion which is reverse to Condition (b) in
Theorem \ref{main-pure-states}. This is another way of showing \eqref{samezeroset} in the following remark.

\begin{rem}
Suppose that we are again in the situation of Theorem \ref{main-pure-states}.
From the conclusion of this theorem, we have that there exists $\ep >0$ such that $f-\ep u \in M$.
By Lemma \ref{prop-unit}, Condition (d) implies that $u$ is even $I$-stable contained in $M$. In particular, there is $\ep>0$ such that
$u-\ep f\in M$. Together this implies a geometric fact that is even
stronger than \eqref{samezeroset}, namely that there exists $\ep>0$ such that
\begin{align}\label{samebehaviour}
f\ge\ep u\text{ on }S(M)\qquad\text{and}\qquad u\ge\ep f\text{ on }S(M).
\end{align}
This implies in particular
\begin{align}\label{samezeroset}
Z:=Z(f)\cap S(M)=Z(u)\cap S(M).
\end{align}
Since $S(M)$ is compact (since it is closed and by Proposition \ref{archchar} bounded), \eqref{samebehaviour} is for any open set $U\subseteq\R^n$
containing $Z$ equivalent to
\begin{align}\label{nearzeros}
f\ge\ep u\text{ on }U\cap S(M)\qquad\text{and}\qquad u\ge\ep f\text{ on }U\cap S(M).
\end{align}
This means that in order for fulfilling the hypotheses (and thus the conclusion) of Theorem \ref{main-pure-states},
$f$ and $u$ not only need to have the same zeros on $S(M)$ but moreover have to behave essentially similar near their zeros on $S(M)$. This can give a good hint on how
choose $u$.
\end{rem}

Suppose we are given an Archimedean quadratic module
$M$ of $\R[x]$ and a polynomial $f\in\R[x]$ with $f\ge0$ on $S(M)$. We end this section by proposing a step-by-step
strategy by which one could try to prove $f\in M$ using Theorem~\ref{main-pure-states}:
\begin{enumerate}[Step 1.]
\item Find some $u\in M$ with $uM\subseteq M$ that has exactly the same zeros on $S(M)$ as $f$ and moreover behaves similar near these
zeros, that is, such that \eqref{nearzeros} holds. If $f$ is composed in a certain way from other polynomials which we call its \emph{constituents},
this could perhaps be done by carefully introducing new terms or modifying some terms in this composition that results in a different polynomial $u$ that for some
reason is known to be in $M$.
\item Identify an (often finite and small) nonempty set $F\subseteq \R[x]$ of polynomials such that $u$ is even $F$-stably contained in $M$. Often, $F$ will contain
certain constituents from Step 1. The bigger $F$ is, the bigger gets the ideal $I:=I(F)$ and the easier Condition (e) from Theorem \ref{main-pure-states} will be satisfied.
In any case, make sure that $f$ and $u$ lie both in $I$.
\item Now try to prove Condition (e) from Theorem~\ref{main-pure-states} by using the defining property of test states from Definition \ref{def-test-state}. This can often
be done by comparing what happens when the test state $\ph$ is applied to the respective expressions by which $f$ and $u$ are built from the constituents. The hope is that
the fact that $\ph(u)=1$ would imply positivity of $\ph$ on certain subexpressions (make sure they lie in the ideal $I$ in order for $\ph$ to being defined on them)
which in turn would imply $\ph(f)>0$. 
\end{enumerate}
The philosophy between Theorem~\ref{main-pure-states} could be described as follows to a general audience: You have to find a ``role-model element'' $u$ of $M$ that is
for some reason contained in $M$ and in fact even stably contained in some sense. If we have now a polynomial $f$
having similar geometry on $S(M)$ and passing a number of tests
on which the role-model $u$ does well, then $f$ is also an element of $M$. The tests are related to the zeros of $f$ (and at the same time of $u$) on $S(M)$
(which prevent the application of Putinar's theorem) but nevertheless can be of algebraic nature.

In the next two sections we will see concrete examples of how to apply Theorem~\ref{main-pure-states}. In some sense, they will indeed be very simple examples.
The set $F$ will be a
singleton in Section \ref{sec:5} and a two-element set in Section \ref{sec:stable}. We expect the theorem to have much deeper applications in the future.

\section{Certifying copositivity of matrices of size five}\label{sec:5}

This section is devoted to show Theorem \ref{reznick5cop}, namely that every copositive matrix of size $5$ is Reznick-certifiable. For this, we will use the following important result from \cite[Theorem 1.3]{lv3} that reduces Theorem \ref{reznick5cop} to showing that the positive diagonal scalings of the Horn matrix are Reznick-certifiable.

\begin{thm}[Laurent and Vargas]\label{theorem-reduction}
Equality $\COP_5=\bigcup_{r\ge0}\MK_5^{(r)}$ holds
if and only if for every positive definite diagonal matrix $D$, the matrix $DHD$ is Reznick-certifiable.
\end{thm} 

In order to show that every positive diagonal scaling of the Horn matrix is Reznick-certifiable, we observe the following. If $d_1,\ldots, d_5\in\R_{>0}$ are the diagonal entries
of the diagonal matrix $D\in\SYM_5$, then
\begin{align*}
DHD \text{ is Reznick-certifiable } & \iff (x^{\circ2})^TDHDx^{\circ2} \in M_{\S^4}\\
						   & \iff (x^{\circ2})^THx^{\circ2} \in \Si  + I\left(\sum_{i=1}^5\frac{1}{d_i}x_i^2-1\right)
\end{align*}
where the first equivalence holds by Corollary \ref{equiv-cert} and the second one follows from the variable substitutions $x_i\mapsto\frac1{\sqrt{d_i}}x_i$ and
$x_i\mapsto\sqrt{d_i}{x_i}$.
This together with Theorem~\ref{theorem-reduction} reduces Theorem \ref{reznick5cop} to the following.

\begin{thm}\label{theorem-horn}
Let $d_1,\dots, d_5\in\R_{>0}$. Then
\[(x^{\circ2})^THx^{\circ2} \in \Si  + I\left(\sum_{i=1}^5d_ix_i^2-1\right).\]
Equivalently, every positive diagonal scaling of the Horn matrix is Reznick-certifiable.
\end{thm}

To prove this, we will apply Theorem \ref{main-pure-states} in a special setting. We start with a preliminary result from \cite{lv2} that will be used in the proof of Theorem~\ref{theorem-horn}. This result was originally stated as a characterization of the diagonal scalings of the Horn matrix that belong to the cone $\MK_5^{(1)}$
\cite[Theorem 4]{lv2}. We use the following reformulation of it.

\begin{lem}\label{lemma-DHD}
Let $d_1,\ldots, d_5\in\R_{>0}$. Then
\[\left(\sum_{i=1}^{5}d_ix_i^2\right)(x^{\circ2})^T Hx^{\circ2} \in \Si\iff d_{i-1}+d_{i+1} \ge d_i \text{ for } i\in [5]\]
where the indices have to be understood modulo $5$.
\end{lem}

We will just use the ``$\Longleftarrow$" part of  Lemma \ref{lemma-DHD} that also follows from the following explicit decomposition (which follows from \cite{lv2}
and generalizes the decomposition from \cite[Page 68]{par}):
\begin{align*}
(\sum_{i=1}^5d_ix_i^2)(x^{\circ2})^THx^{\circ2}=\ &d_1x_1^2(x_1^2+x_2^2+x_5^2-x_3^2-x_4^2)^2+\\
&d_2x_2^2(x_1^2+x_2^2+x_3^2-x_4^2-x_5^2)^2+\\
&d_3x_3^2(x_2^2+ x_3^2+x_4^2-x_5^2-x_1^2)^2+\\
&d_4x_4^2(x_3^2+x_4^2+x_5^2-x_1^2-x_2^2)^2+\\
&d_5x_5^2(x_1^2+x_4^2+x_5^2-x_2^2-x_3^2)^2+\\
&4x_1^2x_2^2x_5^2(d_5-d_1+d_2) + 4x_1^2x_2^2x_3^2(d_3+d_1-d_2)+\\
& 4x_2^2x_3^2x_4^2(d_4+d_2-d_3) + 4x_3^2x_4^2x_5^2(d_5+d_3-d_4)+\\
&4x_4^2x_5^2x_1^2(d_1+d_4-d_5).
\end{align*}
In particular, if $(d_1,\ldots, d_5)$ is sufficiently close to the all ones vector, then \[\left(\sum_{i=1}^{5}d_ix_i^2\right)(x^{\circ2})^T Hx^{\circ2}\] is a sum of squares.
This is exactly what we will use in the proof of Theorem \ref{theorem-horn} which now follows.

\begin{proof}[Proof of Theorem \ref{theorem-horn}]
Set $h:=(x^{\circ2})^THx^{\circ2}$. We will show $h\in \Si + I(\sum_{i=1}^5d_ix_i^2-1)$ by applying Theorem \ref{main-pure-states} in the following setting: 
\begin{itemize}
\item $F:=\{h\}$ generates the ideal $I:=\R[x]h$,
\item  $M := \Si + I\left(\sum_{i=1}^5d_ix_i^2-1\right)$,
\item $f:=h$, and
\item  $u:=\left(\sum_{i=1}^5x_i^2\right)h$.
\end{itemize}
In what follows we will show that this setting satisfies the hypotheses of Theorem~\ref{main-pure-states}, thus enabling us to conclude that $h \in M$, as desired.

 First,  we show that $M$ is Archimedean. We have $1-\sum_{i=1}^5d_ix_i^2\in M$. If we set $d:=\min\{d_i\mid i\in[5]\}$, then we have $1-\sum_{i=1}^5dx_i^2\in M$, so that $\frac{1}{d}-\sum_{i=1}^5x_i^2\in M$. Thus, for any $N>\frac{1}{d}$, we have $N-\sum_{i=1}^5x_i^2\in M$ so that Proposition \ref{archchar} applies.

Since $H$ is copositive we have that $h\ge0$ on $\R^5$. In particular, $h\ge 0$ on $S(M)$. 

Clearly, we have $Z(h) \cap S(M) \subseteq Z(u)\cap S(M)$,  and $uM\subseteq M$ holds as $u\in \Si$ (since $H\in \MK_5^{(1)}$).
 
We now show that $u$ is $F$-stably contained in $M$. By Lemma \ref{lemma-DHD}, the two
polynomials
\[
\left(\sum_{i=1}^5x_i^2\pm \ep \sum_{i=1}^5d_ix_i^2\right)h = \left(\sum_{i=1}^5x_i^2\right)h \pm \ep \left(\sum_{i=1}^5d_ix_i^2-1+1\right)h
\]
are sums of squares for some $\ep>0$ small enough. 
This implies that \[(\sum_{i=1}^5x_i^2)h\pm \ep h\in M\] as wished.

It only remains to show that for all $a\in Z(f)\cap S(M)$ and all test states $\ph\in T(I, M, u)_a$  we have $\ph(h)>0$.
But for such $a$ and $\ph$, we have
\[1=\ph(u)=\left(\sum_{i=1}^5a_i^2\right)\ph(h).\]
by the properties of a test state from Definition \ref{def-test-state}. This forces
$\ph(h)>0$ since $\sum_{i=1}^5a_i^2\ge 0$.
\end{proof}

\section{Finite convergence of the hierarchy computing the stability number}\label{sec:stable}

In this section, we show the second main application of this paper (Theorem~\ref{theorem-alpha}): Given a graph $G=([n],E)$, the hierarchy $\th^{(r)}(G)$ has finite convergence to $\al(G)$. In other words, for every graph $G$ the graph matrix $M_G$ is Reznick-certifiable, i.e., $(\sum_{i=1}^nx_i^2)^rf_G$ is a sum of squares for some $r\in \N_0$.

For any graph $G$ and any $i\notin G$, we denote by $G\oplus i$ the graph that arises from $G$ by adding an isolated node $i$. In other words, if $G=(V,E)$, then
\[G\oplus i:=(V\cup\{i\},E).\]

Since the empty graph is trivially Reznick-certifiable, it would of course be enough to prove that Reznick-certifiability is preserved under adding a node
\emph{together with edges edges connecting it to the existing nodes}.
We use the following much stronger result, which is a reformulation of a result from  from \cite[Proposition~4]{lv2}. It says that it is actually enough to prove that
Reznick-certifiability is preserved under adding an \emph{isolated} node.

\begin{lem}[Laurent and Vargas]\label{lemma-isolated}
Suppose that the following implication holds for all graphs $G=([n],E)$:
\begin{align*}
M_G \text{ is Reznick-certifiable} \implies M_{G\oplus(n+1)} \text{ is Reznick-certifiable}.
\end{align*}
Then, $M_G$ is Reznick-certifiable for all graphs $G=([n],E)$.
\end{lem}

We will prove the following result, which combined with Lemma \ref{lemma-isolated} implies Theorem \ref{theorem-alpha}. This forms the main technical part of this section.

\begin{lem}\label{isolated-nodes}
Let $G=([n], E)$ be a graph. Suppose that $M_G$ is Reznick-certifiable, i.e., \[\left(\sum_{i=1}^nx_i^2\right)^rf_G\in\sum\R[x_1,\ldots,x_n]^2\] for some $r\in \N_0$.
Then, $M_{G\oplus(n+1)}$ is Reznick-certifiable, i.e., \[\left(\sum_{i=1}^{n+1}x_i^2\right)^rf_G\in\sum\R[x_1,\ldots,x_{n+1}]^2\] for some $r\in \N_0$.
\end{lem}

\begin{proof}
By Corollary \ref{equiv-cert}, it suffices to show that
\begin{align*}
 f_{G\oplus(n+1)}\in M_{\S^n}=\sum\R[x_1,\ldots,x_{n+1}]^2+I\left(\sum_{i=1}^{n+1}x_i^2-1\right).
\end{align*}
Set $\al:=\al(G)$, so that $\al(G\oplus(n+1))=\al+1$. Observe that the following identity (which follows also from \cite[Section 3.2]{gl}) holds:
\begin{align}\label{relation-f-g}
 f_{G\oplus(n+1)}=g^2+\frac{\al+1}\al f_G, \text{ where } g:=\sqrt\al x_{n+1}^2-\frac1{\sqrt{\al}}(x_1^2+\ldots+x_n^2).
 \end{align} 
Indeed, we compare coefficients 
\begin{align*}
x_{n+1}^4&: \al= \al \\
x_i^4 \text{ for }(i\ne n+1)&: \al= \frac{1}{\al}+\frac{\al+1}{\al}\cdot(\al-1) \\
x_i^2x_j^2 \text{ for } \{i,j\}\in E&: 2\al= \frac{2}{\al} + \frac{\al+1}{\al}\cdot2(\al-1)\\
x_i^2x_j^2 \text{ for } \{i,j\}\notin E, i,j\ne n+1 &: -2=\frac{2}{\al}-2\cdot\frac{\al+1}{\al}\\
x_{n+1}^2x_i^2 \text{ for } i\ne n+1 &: -2 = -2\cdot\frac{\sqrt{\al}}{\sqrt{\al}}
\end{align*}
We apply Theorem \ref{main-pure-states} in the following setting to the polynomial ring $\R[x_1,\ldots,x_{n+1}]$:
\begin{itemize}
\item $F:= \{g^2,f_G\},$ 
\item $M:=M_{\S^n},$
\item $u:=\underbrace{g^2}_{=:u_1}+\underbrace{\frac{\al+1}\al\left(\sum_{i=1}^nx_i^2\right)^{2r}f_G}_{=:u_2}$, and
\item $f:=f_{G\oplus(n+1)}.$
\end{itemize}
Then, $M$ is Archimedean, $f\in I:=I(F)$ by the identity (\ref{relation-f-g}), and $u\in I$. We now verify the conditions (a)-(e) of Theorem \ref{main-pure-states}.

Clearly, $f_{G\oplus(n+1)}\ge 0$ on $S(M)$, because $f_{G\oplus(n+1)}$ is positive semidefinite (as $M_{G\oplus(n+1)}$ is copositive). Hence Condition (a) is verified.
Now, by looking at identity (\ref{relation-f-g}), if $a\in\R^{n+1}$ satisfies $f_{G\oplus(n+1)}(a)=0$, then $g^2(a)=0$ and $f_G(a)=0$. This implies 
\[Z(f_{G\oplus(n+1)})\subseteq Z(u),\]
and thus (b) holds.
The inclusion $uM\subseteq M$ holds as $u\in\sum\R[x_1,\ldots,x_{n+1}]^2$ by construction. This is condition (c).

Next, we show that $u$ is $F$-stably contained in $M$, which is (d). First, it is clear that $u\pm g^2$ is a sum of squares, so it belongs to $M$. It remains to prove that there exists $\ep>0$ such that $u\pm \ep f_G\in M$, which is equivalent to show that there exists $N>0$ such that $Nu\pm f_G\in M$. For this, we will show the following two statements.

\begin{enumerate}[(+)]
\item[(+)] There exist $N_1,N_2\in\N$ such that $N_1u_1+N_2u_2+f_G\in M$, 
\item[(-)] There exist $N_1,N_2\in\N$ such that $N_1u_1+N_2u_2-f_G\in M$,
\end{enumerate}
If this holds, then using that $u_1,u_2\in \Si\subseteq M$, we obtain that there exists $N\in \N$ such that $Nu\pm f_G\in M$, as desired.

For proving (+) and (-), let $\equiv$ denote congruence modulo the ideal \[I\left(\sum_{i=1}^{n+1}x_i^2-1\right),\] i.e., $p\equiv q$ means that the difference $p-q$ lies
in that ideal for $p,q\in\R[x_1,\ldots,x_{n+1}]$.
Observe that $p\in M$ if and only if $q\in M$ whenever $p\equiv q$. We have
\begin{align}
1-x_{n+1}^2&\equiv \sum_{i=1}^n x_i^2 \label{equiv-2}\quad\text{and}\\
g&\equiv \frac{1}{\sqrt{\al}}((\al+1)x_{n+1}^2-1). \label{equiv-1}
\end{align}

\textbf{Proof of (+).}
Consider the univariate polynomial \[p:=c'(1-x_{n+1}^2)^{r}-1\in\R[x_{n+1}],\] where $c':=(1-\frac{1}{\al+1})^{-r}$. Since $\pm \frac{1}{\sqrt{\al+1}}$ are roots of $p$,
\[(\al+1)x_{n+1}^2-1\] divides $p$ in $\R[x_{n+1}]$, so we can write 
\[p= ((\al+1)x_{n+1}^2-1)q\]
for some $q\in \R[x_{n+1}]$.  Since $M$ is Archimedean, using Definition \ref{def-qm}(c), there exists $C\in \N$ such that 
\[C+q^2f_G\in M.\]
Since $M$ is a quadratic module, it follows that
\[C((\al+1)x_{n+1}^2-1)^2 + p^2f_G =((\al+1)x_{n+1}^2-1)^2(C+q^2f_G)\in M.\]
Then, by using the definition of $p$, we obtain 
\[C((\al+1)x_{n+1}^2-1)^2 + c'^2(1-x_{n+1}^2)^{2r}f_G-2c'(1-x_{n+1}^2)^{r}f_G + f_G\in M.\]
Using (\ref{equiv-1}) and (\ref{equiv-2}) we obtain
\[\al Cg^2+c'^2\left(\sum_{i=1}^nx_i^2\right)^{2r}f_G -2c'\left(\sum_{i=1}^nx_i^2\right)^{r}f_G+f_G\in M.\]
By assumption, we have that $\left(\sum_{i=1}^nx_i^2\right)^{r}f_G\in \Si\subseteq M$ and thus 
\[\al Cg^2+c'^2\left(\sum_{i=1}^nx_i^2\right)^{2r}f_G +f_G\in M,\] 
which shows (+).

\bigskip
\textbf{Proof of (-).}
Consider the univariate polynomial \[p:=c'(1-x_{n+1}^2)^{2r}-1\in\R[x_{n+1}],\] where $c':=(1-\frac{1}{\al+1})^{-2r}$. Since $\pm \frac{1}{\sqrt{\al+1}}$ are roots of $p$,
\[(\al+1)x_{n+1}^2-1\] divides $p$ in $\R[x_{n+1}]$, so we can write 
\[p= ((\al+1)x_{n+1}^2-1)q\]
for some $q\in \R[x_{n+1}]$.  Since $M$ is Archimedean, there exists $C\in \N$ such that 
\[C-q^2f_G\in M.\]
Since $M$ is a quadratic module, it follows that
\[C((\al+1)x_{n+1}^2-1)^2 - p^2f_G = ((\al+1)x_{n+1}^2-1)^2(C-q^2f_G) \in M.\]
That is,
\[ C((\al+1)x_{n+1}^2-1)^2 - c'^2(1-x_{n+1}^2)^{4r}f_G+2c'(1-x_{n+1}^2)^{2r}f_G - f_G\in M.\]
Using (\ref{equiv-1}) and (\ref{equiv-2}), we obtain
\[ \al Cg^2- c'^2\left(\sum_{i=1}^nx_i^2\right)^{4r}f_G +2c'\left(\sum_{i=1}^nx_i^2\right)^{2r}f_G-f_G\in M. \]
By assumption, we have $\left(\sum_{i=1}^nx_i^2\right)^{r}f_G\in \Si$. This implies $(\sum_{i=1}^nx_i^2)^{4r}f_G\in \Si$. Hence, we have 
\[ \al Cg^2+2c'\left(\sum_{i=1}^nx_i^2\right)^{2r}f_G-f_G\in M,\]
which shows (-).

Hence we have verified Condition (d).

Finally, we check the test state property (e). Let $\ph$ be a test state on $I$ for $M$ at a point  $a\in Z(f_{G\oplus(n+1)})\cap S(M)$ with respect to $u$. Since $(\sum_{i=1}^nx_i^2)^{r}f_G\in M\cap I$, we have that 
\[0\le \ph\left(\left(\sum_{i=1}^nx_i^2\right)^{r}f_G\right)=\left(\sum_{i=1}^na_i^2\right)^{r}\ph(f_G),\]
where $a\in\S^n$ (recall Remark \ref{rem-support}) and $f_{G\oplus(n+1)}(a)=0$. It is easy to observe that $f_{G\oplus(n+1)}(0,\dots,0,\pm1)>0$, so that
$a\ne(0,\ldots,0,\pm1)$. This implies $\sum_{i=1}^na_i^2>0$, and thus $\ph(f_G)\ge 0$. Since $g^2\in I\cap M$, we have that $\ph(g^2)\ge 0$. Also, we have
\begin{align*}
1=\ph(u)&=\ph(g^2)+\frac{\al+1}{\al}\ph\left(\left(\sum_{i=1}^nx_i^2\right)^{2r}f_G\right)\\
&=\ph(g^2)+\frac{\al+1}{\al}\left(\sum_{i=1}^na_i^2\right)^{2r}\ph(f_G).
\end{align*}
Therefore, $\ph(g^2)$ and $\ph(f_G)$ are nonnegative but they cannot be both zero. Using relation (\ref{relation-f-g}), we obtain
\[ \ph(f_{G\oplus(n+1)})= \ph(g^2) + \frac{\al+1}{\al}\ph(f_G)>0,\]
as desired.
\end{proof}
\subsection*{Acknowledgements}
We thank Monique Laurent for her valuable comments about the presentation of this paper. This work is supported by the European Union's Framework Programme for Research and Innovation Horizon
2020 under the Marie Skłodowska-Curie Actions Grant Agreement No. 813211  (POEMA).


\begin{thebibliography}{DDGH}

\bibitem[Art]{art}
E. Artin:
Über die Zerlegung definiter Funktionen in Quadrate,
Abh. Math. Sem. Univ. Hamburg 5 (1927), no.1, 100--115

\bibitem[BA]{ba}
J.-B. Bru, Jean-Bernard, W. Alberto de Siqueira Pedra:
C*-Algebras and Mathematical Foundations of Quantum Statistical Mechanics,
An Introduction,
Lat. Amer. Math. Ser.,
Springer, Cham, 2023

\bibitem[BK]{bk}
I.M. Bomze, E. de Klerk:
Solving standard quadratic optimization problems via linear, semidefinite and copositive programming,
J. Global Optim. 24 (2002), no.2, 163--185

\bibitem[BKT]{bkt}
M. Bodirsky, M. Kummer, A. Thom:
Spectrahedral shadows and completely positive maps on real closed fields,
preprint [\url{https://arxiv.org/abs/2206.06312}]

\bibitem[Ble]{ble} G. Blekherman:
There are significantly more nonnegative polynomials than sums of squares,
Israel J. Math. 153 (2006), 355--380

\bibitem[BSS]{bss}
S. Burgdorf, C. Scheiderer, M. Schweighofer:
Pure states, nonnegative polynomials and sums of squares,
Comment. Math. Helv. 87 (2012), no. 1, 113--140

\bibitem[Bur]{bur}
S. Burer:
On the copositive representation of binary and continuous nonconvex quadratic programs,
Math. Program. 120, Ser. A (2009), no. 2, 479--495

\bibitem[Cas]{cas}
G. Cassier:
Problème des moments sur un compact de Rn et décomposition de polynômes à plusieurs variables,
J. Funct. Anal. 58 (1984), no. 3, 254--266

\bibitem[CL]{cl}
M.D. Choi, T.Y. Lam:
An old question of Hilbert,
Conference on Quadratic Forms 1976,
Proc. Conf., Queen's Univ., Kingston, Ont., 1976, 385--405

\bibitem[DDGH]{ddgh}
P. Dickinson,  M. Dür, L. Gijben, R. Hildebrand:
Scaling relationship between the copositive cone and Parrilo's first level approximation,
Optim. Lett. 7 (2013), no. 8, 1669--1679

\bibitem[Dia]{dia}
P. Diananda:
On non-negative forms in real variables some or all of which are non-negative,
Proc. Cambridge Philos. Soc. 58 (1962), 17--25

\bibitem[EHS]{EHS}
E.G. Effros, D.E. Handelman, C.L. Shen:
Dimension groups and their affine representations,
Amer. J. Math. 102 (1980), no. 2, 385--407

\bibitem[GL]{gl}
N. Gvozdenović, M. Laurent:
Semidefinite bounds for the stability number of a graph via sums of squares of polynomials,
Math. Program. 110, Ser. B (2007), no. 1, 145--173

\bibitem[Han]{han}
D. Handelman,:
Positive polynomials and product type actions of compact groups,
Mem. Amer. Math. Soc. 54 (1985), no. 320

\bibitem[Hilb]{hilb}
D. Hilbert:
Ueber die Darstellung definiter Formen als Summe von Formenquadraten,
Math. Ann. 32 (1888), no. 3, 342--350

\bibitem[Hild]{hild}
R. Hildebrand:
The extreme rays of the $5\times5$ copositive cone,
Linear Algebra Appl. 437 (2012), no. 7, 1538--1547

\bibitem[HN]{hn}
M. Hall, M. Newman,
Copositive and completely positive quadratic forms,
Proc. Cambridge Philos. Soc. 59 (1963), 329--339

\bibitem[Jac]{jac}
T. Jacobi:
A representation theorem for certain partially ordered commutative rings.,
Math. Z. 237 (2001), no. 2, 259--273

\bibitem[Kar]{kar}
R. Karp:
Reducibility among combinatorial problems,
Complexity of computer computations, Proc. Sympos., IBM Thomas J. Watson Res. Center, Yorktown Heights, N.Y., 1972, 85--103

\bibitem[KLP]{klp}
E. de Klerk, M. Laurent, P. Parrilo:
On the equivalence of algebraic approaches to the minimization of forms on the simplex,
Positive polynomials in control, 121--132,
Lect. Notes Control Inf. Sci. 312,
Springer-Verlag, Berlin, 2005

\bibitem[KP]{kp}
E. de Klerk, D.V. Pasechnik:
Approximation of the stability number of a graph via copositive programming,
SIAM J. Optim. 12 (2002), no. 4, 875--892

\bibitem[Kri]{kri}
J.-L. Krivine:
Anneaux préordonnés,
J. Analyse Math. 12 (1964), 307--326

\bibitem[LV1]{lv1}
M. Laurent, L.F. Vargas:
Finite convergence of sum-of-squares hierarchies for the stability number of a graph,
SIAM J. Optim. 32 (2022), no. 2, 491--518

\bibitem[LV2]{lv2}
M. Laurent, L.F. Vargas:
Exactness of Parrilo's conic approximations for copositive matrices and associated low order bounds for the stability number of a graph,
Math. Oper. Res. 48 (2023), no. 2, 1017--1043

\bibitem[LV3]{lv3}
M. Laurent, L.F. Vargas:
On the exactness of sum-of-squares approximations for the cone of $5\times5$ copositive matrices,
Linear Algebra Appl. 651 (2022), 26--50


\bibitem[MK]{MK}
K.G. Murty, S.N. Kabadi:
Some NP-complete problems in quadratic and nonlinear programming,
Math. Programming 39 (1987), no. 2, 117--129

\bibitem[MS]{ms}
T.S. Motzkin, E.G. Straus:
Maxima for graphs and a new proof of a theorem of Turán,
Canadian J. Math. 17 (1965), 533--540

\bibitem[Nie]{nie}
J. Nie:
Optimality conditions and finite convergence of Lasserre's hierarchy,
Math. Program. 146, Ser. A (2014), no. 1--2, 97--121

\bibitem[Par]{par}
P.A. Parrilo:
Structured semidefinite programs and semialgebraic geometry methods in robustness and optimization,
PhD thesis, California Institute of Technology, 2000
[\url{https://thesis.library.caltech.edu/1647/1/Parrilo-Thesis.pdf}]

\bibitem[Put]{put}
M. Putinar:
Positive polynomials on compact semi-algebraic sets,
Indiana Univ. Math. J. 42 (1993), no. 3, 969--984

\bibitem[PVZ]{pvz}
J. Peña, J. Vera, L.F. Zuluaga:
Computing the stability number of a graph via linear and semidefinite programming,
SIAM J. Optim.18 (2007), no. 1, 87--105 

\bibitem[Rez]{rez}
B. Reznick:
Uniform denominators in Hilbert's seventeenth problem,
Math. Z. 220 (1995), no. 1, 75--97

\bibitem[Sch1]{sch1}
C. Scheiderer:
Sums of squares on real algebraic surfaces,
Manuscripta Math. 119 (2006), no. 4, 395--410

\bibitem[Sch2]{sch2}
C. Scheiderer:
Positivity and sums of squares: a guide to recent results,
Emerging applications of algebraic geometry, 271--324,
IMA Vol. Math. Appl., 149, Springer, New York, 2009

\bibitem[Schw]{schw}
M. Schweighofer.
Real algebraic geometry, positivity and convexity,
lecture notes, preprint [\url{https://arxiv.org/abs/2205.04211}]

\bibitem[Var]{var}
L.F. Vargas:
Sum-of-squares representations for copositive matrices and independent sets in graphs,
PhD Thesis, Tilburg University, 2023

\bibitem[VL]{vl}
L.F. Vargas, M. Laurent:
Copositive matrices, sums of squares and the stability number of a graph, In: Polynomial Optimization, Moments, and Applications. M. Kocvara, B. Mourrain, C. Riener (eds). Springer Optimization and Its Applications (SOIA, volume 206) (2023), 99--132.
\end{thebibliography}
\end{document}